\documentclass[11pt]{article}
\usepackage{epsf}
\usepackage{amsmath}
\usepackage{epsfig}
\usepackage{times}
\usepackage{amssymb}
\usepackage{amsthm}
\usepackage{setspace}
\usepackage{cite}

\usepackage{algorithmic}  
\usepackage{algorithm}

\usepackage{shadow}
\usepackage{fancybox}
\usepackage{fancyhdr}

\def\y{{\bf y}}

\def\x{{\bf x}}

\def\x{{\mathbf x}}

\def\x{{\bf x}}
\def\y{{\bf y}}

\def\h{{\bf h}}

\def\cH{{\cal H}}

\def\be{\begin{equation}}
\def\ee{\end{equation}}
\def\ba{\left[\begin{array}}
\def\ea{\end{array}\right]}

\def\x{{\bf x}}
\def\y{{\bf y}}

\def\1{{\bf 1}}

\def\g{{\bf g}}
\def\0{{\bf 0}}

\newtheorem{theorem}{Theorem}

\newtheorem{lemma}{Lemma}

\begin{document}

\begin{singlespace}

\title {Lifting/lowering Hopfield models ground state energies  
}
\author{
\textsc{Mihailo Stojnic}
\\
\\
{School of Industrial Engineering}\\
{Purdue University, West Lafayette, IN 47907} \\
{e-mail: {\tt mstojnic@purdue.edu}} }
\date{}
\maketitle

\centerline{{\bf Abstract}} \vspace*{0.1in}

In our recent work \cite{StojnicHopBnds10} we looked at a class of random optimization problems that arise in the forms typically known as Hopfield models. We viewed two scenarios which we termed as the positive Hopfield form and the negative Hopfield form. For both of these scenarios we defined the binary optimization problems whose optimal values essentially emulate what would typically be known as the ground state energy of these models. We then presented a simple mechanisms that can be used to create a set of theoretical rigorous bounds for these energies. In this paper we create a way more powerful set of mechanisms that can substantially improve the simple bounds given in \cite{StojnicHopBnds10}. In fact, the mechanisms we create in this paper are the first set of results that show that convexity type of bounds can be substantially improved in this type of combinatorial problems.

\vspace*{0.25in} \noindent {\bf Index Terms: Hopfield models; ground-state energy}.

\end{singlespace}

\section{Introduction}
\label{sec:back}

We start by recalling on what the Hopfield models are. These models are well known in mathematical physics. However, we will be purely interested in their mathematical properties and the definitions that we will give below in this short introduction are almost completely mathematical without much of physics type of insight. Given a relatively simple and above all well known structure of the Hopfield models we do believe that it will be fairly easy for readers from both, mathematics and physics, communities to connect to the parts they find important to them. Before proceeding with the detailed introduction of the models we will also mention that fairly often we will define mathematical objects of interest but will occasionally refer to them using their names typically known in physics.

The model that we will study was popularized in \cite{Hop82} (or if viewed in a different context one could say in \cite{PasFig78,Hebb49}). It essentially looks at what is called Hamiltonian of the following type
\begin{equation}
\cH(H,\x)=\sum_{i\neq j}^{n}  A_{ij}\x_i\x_j,\label{eq:ham}
\end{equation}
where
\begin{equation}
A_{ij}(H)=\sum_{l=1}^{m}  H_{il}H_{lj},\label{eq:hamAij}
\end{equation}
is the so-called quenched interaction and $H$ is an $m\times n$ matrix that can be also viewed as the matrix of the so-called stored patterns (we will typically consider scenario where $m$ and $n$ are large and $\frac{m}{n}=\alpha$ where $\alpha$ is a constant independent of $n$; however, many of our results will hold even for fixed $m$ and $n$). Each pattern is essentially a row of matrix $H$ while vector $\x$ is a vector from $R^n$ that emulates spins (or in a different context one may say neuron states). Typically, one assumes that the patterns are binary and that each neuron can have two states (spins) and hence the elements of matrix $H$ as well as elements of vector $\x$ are typically assumed to be from set $\{-\frac{1}{\sqrt{n}},\frac{1}{\sqrt{n}}\}$. In physics literature one usually follows convention and introduces a minus sign in front of the Hamiltonian given in (\ref{eq:ham}). Since our main concern is not really the physical interpretation of the given Hamiltonian but rather mathematical properties of such forms we will avoid the minus sign and keep the form as in (\ref{eq:ham}).

To characterize the behavior of physical interpretations that can be described through the above Hamiltonian one then looks at the partition function
\begin{equation}
Z(\beta,H)=\sum_{\x\in\{-\frac{1}{\sqrt{n}},\frac{1}{\sqrt{n}}\}^n}e^{\beta\cH(H,\x)},\label{eq:partfun}
\end{equation}
where $\beta>0$ is what is typically called the inverse temperature. Depending of what is the interest of studying one can then also look at a more appropriate scaled $\log$ version of $Z(\beta,H)$ (typically called the free energy)
\begin{equation}
f_p(n,\beta,H)=\frac{\log{(Z(\beta,H)})}{\beta n}=\frac{\log{(\sum_{\x\in\{-\frac{1}{\sqrt{n}},\frac{1}{\sqrt{n}}\}^n}e^{\beta\cH(H,\x)})}}{\beta n}.\label{eq:logpartfun}
\end{equation}
Studying behavior of the partition function or the free energy of the Hopfield model of course has a long history. Since we will not focus on the entire function in this paper we just briefly mention that a long line of results can be found in e.g. excellent references \cite{PasShchTir94,ShchTir93,BarGenGueTan10,BarGenGueTan12,Tal98}. In this paper we will focus on studying optimization/algorithmic aspects of $\frac{\log{(Z(\beta,H)})}{\beta n}$. More specifically, we will look at a particular regime $\beta,n\rightarrow\infty$ (which is typically called a zero-temperature thermodynamic limit regime or as we will occasionally call it the ground state regime). In such a regime one has
\begin{equation}
\hspace{-.3in}\lim_{\beta,n\rightarrow\infty}f_p(n,\beta,H)=
\lim_{\beta,n\rightarrow\infty}\frac{\log{(Z(\beta,H)})}{\beta n}=\lim_{n\rightarrow\infty}\frac{\max_{\x\in\{-\frac{1}{\sqrt{n}},\frac{1}{\sqrt{n}}\}^n}\cH(H,\x)}{n}
=\lim_{n\rightarrow\infty}\frac{\max_{\x\in\{-\frac{1}{\sqrt{n}},\frac{1}{\sqrt{n}}\}^n}\|H\x\|_2^2}{n},\label{eq:limlogpartfun}
\end{equation}
which essentially renders the following form (often called the ground state energy)
\begin{equation}
\lim_{\beta,n\rightarrow\infty}f_p(n,\beta,H)=\lim_{n\rightarrow\infty}\frac{\max_{\x\in\{-\frac{1}{\sqrt{n}},\frac{1}{\sqrt{n}}\}^n}\|H\x\|_2^2}{n},\label{eq:posham}
\end{equation}
which will be one of the main subjects that we study in this paper. We will refer to the optimization part of (\ref{eq:posham}) as the positive Hopfield form.

In addition to this form we will also study its a negative counterpart. Namely, instead of the partition function given in (\ref{eq:partfun}) one can look at a corresponding partition function of a negative Hamiltonian from (\ref{eq:ham}) (alternatively, one can say that instead of looking at the partition function defined for positive temperatures/inverse temperatures one can also look at the corresponding partition function defined for negative temperatures/inverse temperatures). In that case (\ref{eq:partfun}) becomes
\begin{equation}
Z(\beta,H)=\sum_{\x\in\{-\frac{1}{\sqrt{n}},\frac{1}{\sqrt{n}}\}^n}e^{-\beta\cH(H,\x)},\label{eq:partfunneg}
\end{equation}
and if one then looks at its an analogue to (\ref{eq:limlogpartfun}) one then obtains
\begin{equation}
\hspace{-.3in}\lim_{\beta,n\rightarrow\infty}f_n(n,\beta,H)=\lim_{\beta,n\rightarrow\infty}\frac{\log{(Z(\beta,H)})}{\beta n}=\lim_{n\rightarrow\infty}\frac{\max_{\x\in\{-\frac{1}{\sqrt{n}},\frac{1}{\sqrt{n}}\}^n}-\cH(H,\x)}{n}
=\lim_{n\rightarrow\infty}\frac{\min_{\x\in\{-\frac{1}{\sqrt{n}},\frac{1}{\sqrt{n}}\}^n}\|H\x\|_2^2}{n}.\label{eq:limlogpartfunneg}
\end{equation}
This then ultimately renders the following form which is in a way a negative counterpart to (\ref{eq:posham})
\begin{equation}
\lim_{\beta,n\rightarrow\infty}f_n(n,\beta,H)=\lim_{n\rightarrow\infty}\frac{\min_{\x\in\{-\frac{1}{\sqrt{n}},\frac{1}{\sqrt{n}}\}^n}\|H\x\|_2^2}{n}.\label{eq:negham}
\end{equation}
We will then correspondingly refer to the optimization part of (\ref{eq:negham}) as the negative Hopfield form.

In the following sections we will present a collection of results that relate to behavior of the forms given in (\ref{eq:posham}) and (\ref{eq:negham}) when they are viewed in a statistical scenario. The results that we will present will essentially correspond to what is called the ground state energies of these models. As it will turn out, in the statistical scenario that we will consider, (\ref{eq:posham}) and (\ref{eq:negham}) will be almost completely characterized by their corresponding average values
\begin{equation}
\lim_{\beta,n\rightarrow\infty}Ef_p(n,\beta,H)=\lim_{n\rightarrow\infty}\frac{E\max_{\x\in\{-\frac{1}{\sqrt{n}},\frac{1}{\sqrt{n}}\}^n}\|H\x\|_2^2}{n}\label{eq:poshamavg}
\end{equation}
and
\begin{equation}
\lim_{\beta,n\rightarrow\infty}Ef_n(n,\beta,H)=\lim_{n\rightarrow\infty}\frac{E\min_{\x\in\{-\frac{1}{\sqrt{n}},\frac{1}{\sqrt{n}}\}^n}\|H\x\|_2^2}{n}.\label{eq:neghamavg}
\end{equation}

Before proceeding further with our presentation we will be a little bit more specific about the organization of the paper. In Section \ref{sec:poshop} we will present a powerful mechanism that can be used create bounds on the ground state energies of the positive Hopfield form in a statistical scenario. We will then in Section \ref{sec:neghop} present the corresponding results for the negative Hopfield form. In Section \ref{sec:conc} we will present a brief discussion and several concluding remarks.

\section{Positive Hopfield form}
\label{sec:poshop}

In this section we will look at the following optimization problem (which clearly is the key component in estimating the ground state energy in the thermodynamic limit)
\begin{equation}
\max_{\x\in\{-\frac{1}{\sqrt{n}},\frac{1}{\sqrt{n}}\}^n}\|H\x\|_2^2.\label{eq:posham1}
\end{equation}
For a deterministic (given fixed) $H$ this problem is of course known to be NP-hard (it essentially falls under the class of binary quadratic optimization problems). Instead of looking at the problem in (\ref{eq:posham1}) in a deterministic way i.e. in a way that assumes that matrix $H$ is deterministic, we will look at it in a statistical scenario (this is of course a typical scenario in statistical physics). Within a framework of statistical physics and neural networks the problem in (\ref{eq:posham1}) is studied assuming that the stored patterns (essentially rows of matrix $H$) are comprised of Bernoulli $\{-1,1\}$ i.i.d. random variables see, e.g. \cite{Tal98,PasShchTir94,ShchTir93}. While our results will turn out to hold in such a scenario as well we will present them in a different scenario: namely, we will assume that the elements of matrix $H$ are i.i.d. standard normals. We will then call the form (\ref{eq:posham1}) with Gaussian $H$, the Gaussian positive Hopfield form. On the other hand, we will call the form (\ref{eq:posham1}) with Bernoulli $H$, the Bernoulli positive Hopfield form. In the remainder of this section we will look at possible ways to estimate the optimal value of the optimization problem in (\ref{eq:posham1}). Below we will introduce a strategy that can be used to obtain an upper bound on the optimal value.

\subsection{Upper-bounding ground state energy of the positive Hopfield form}
\label{sec:poshopub}

In this section we will look at problem from (\ref{eq:posham1}). In fact, to be a bit more precise, in order to make the exposition as simple as possible, we will look at its a slight variant given below
\begin{equation}
\xi_p=\max_{\x\in\{-\frac{1}{\sqrt{n}},\frac{1}{\sqrt{n}}\}^n}\|H\x\|_2.\label{eq:sqrtposham1}
\end{equation}
As mentioned above, we will assume that the elements of $H$ are i.i.d. standard normal random variables. Before proceeding further with the analysis of (\ref{eq:sqrtposham1}) we will recall on several well known results that relate to Gaussian random variables and the processes they create.

First we recall the following results from \cite{Gordon85} that relates to statistical properties of certain Gaussian processes.
\begin{theorem}(\cite{Gordon85})
\label{thm:Gordonpos1} Let $X_{i}$ and $Y_{i}$, $1\leq i\leq n$, be two centered Gaussian processes which satisfy the following inequalities for all choices of indices
\begin{enumerate}
\item $E(X_{i}^2)=E(Y_{i}^2)$
\item $E(X_{i}X_{l})\leq E(Y_{i}Y_{l}), i\neq l$.
\end{enumerate}
Let $\psi()$ be an increasing function on the real axis. Then
\begin{equation*}
E(\min_{i}\psi(X_{i}))\leq E(\min_i \psi(Y_{i})) \Leftrightarrow E(\max_{i}\psi(X_{i}))\geq E(\max_i\psi(Y_{i})).
\end{equation*}
\end{theorem}

In our recent work \cite{StojnicHopBnds10} we rely on the above theorem to create an upper-bound on the ground state energy of the positive Hopfield model. However, the strategy employed in \cite{StojnicHopBnds10} only a basic version of the above theorem where $\psi(x)=x$. Here we will substantially upgrade the strategy by looking at a very simple (but way better) different version of $\psi()$.

We start by reformulating the problem in (\ref{eq:sqrtposham1}) in the following way
\begin{equation}
\xi_p=\max_{\x\in\{-\frac{1}{\sqrt{n}},\frac{1}{\sqrt{n}}\}^n}\max_{\|\y\|_2=1}\y^TH\x.\label{eq:sqrtposham2}
\end{equation}
We do mention without going into details that the ground state energies will concentrate in the thermodynamic limit and hence we will mostly focus on the expected value of $\xi_p$ (one can then easily adapt our results to describe more general probabilistic concentrating properties of ground state energies). The following is then a direct application of Theorem \ref{thm:Gordonpos1}.
\begin{lemma}
Let $H$ be an $m\times n$ matrix with i.i.d. standard normal components. Let $\g$ and $\h$ be $m\times 1$ and $n\times 1$ vectors, respectively, with i.i.d. standard normal components. Also, let $g$ be a standard normal random variable and let $c_3$ be a positive constant. Then
\begin{equation}
E(\max_{\x\in\{-\frac{1}{\sqrt{n}},\frac{1}{\sqrt{n}}\}^n,\|\y\|_2=1}e^{c_3(\y^T H\x + g)})\leq E(\max_{\x\in\{-\frac{1}{\sqrt{n}},\frac{1}{\sqrt{n}}\}^n,\|\y\|_2=1}e^{c_3(\g^T\y+\h^T\x)}).\label{eq:posexplemma}
\end{equation}\label{lemma:posexplemma}
\end{lemma}
\begin{proof}
As mentioned above, the proof is a standard/direct application of Theorem \ref{thm:Gordonpos1}. We will sketch it for completeness. Namely, one starts by defining processes $X_i$ and $Y_i$ in the following way
\begin{equation}
Y_i=(\y^{(i)})^T H\x^{(i)} + g\quad X_i=\g^T\y^{(i)}+\h^T\x^{(i)}.\label{eq:posexplemmaproof1}
\end{equation}
Then clearly
\begin{equation}
EY_i^2=EX_i^2=\|\y^{(i)}\|_2^2+\|\x^{(i)}\|_2^2=2.\label{eq:posexplemmaproof2}
\end{equation}
One then further has
\begin{eqnarray}
EY_iY_l & = & (\y^{(i)})^T\y^{(l)}(\x^{(l)})^T\x^{(i)}+1\nonumber \\
EX_iX_l & = & (\y^{(i)})^T\y^{(l)}+(\x^{(l)})^T\x^{(i)}.\label{eq:posexplemmaproof3}
\end{eqnarray}
And after a small algebraic transformation
\begin{eqnarray}
EY_iY_l-EX_iX_l & = & (1-(\y^{(i)})^T\y^{(l)})-(\x^{(l)})^T\x^{(i)}(1-(\y^{(i)})^T\y^{(l)}) \nonumber \\
& = & (1-(\x^{(l)})^T\x^{(i)})(1-(\y^{(i)})^T\y^{(l)})\nonumber \\
& \geq & 0.\label{eq:posexplemmaproof4}
\end{eqnarray}
Combining (\ref{eq:posexplemmaproof2}) and (\ref{eq:posexplemmaproof4}) and using results of Theorem \ref{thm:Gordonpos1} one then easily obtains (\ref{eq:posexplemma}).
\end{proof}

One then easily has
\begin{multline}
E(e^{c_3(\max_{\x\in\{-\frac{1}{\sqrt{n}},\frac{1}{\sqrt{n}}\}^n}\|H\x\|_2+g)})=
E(e^{c_3(\max_{\x\in\{-\frac{1}{\sqrt{n}},\frac{1}{\sqrt{n}}\}^n}\max_{\|\y\|_2=1}y^TH\x+g)})\\
=E(\max_{\x\in\{-\frac{1}{\sqrt{n}},\frac{1}{\sqrt{n}}\}^n}\max_{\|\y\|_2=1}(e^{c_3(\y^TH\x+g)})).\label{eq:chpos1}
\end{multline}
Connecting (\ref{eq:chpos1}) and results of Lemma \ref{lemma:posexplemma} we have
\begin{multline}
E(e^{c_3(\max_{\x\in\{-\frac{1}{\sqrt{n}},\frac{1}{\sqrt{n}}\}^n}\|H\x\|_2+g)})=
E(\max_{\x\in\{-\frac{1}{\sqrt{n}},\frac{1}{\sqrt{n}}\}^n}\max_{\|\y\|_2=1}(e^{c_3(\y^TH\x+g)}))\\
\leq E(\max_{\x\in\{-\frac{1}{\sqrt{n}},\frac{1}{\sqrt{n}}\}^n}\max_{\|\y\|_2=1}(e^{c_3(\g^T\y+h^T\x)}))
=E(\max_{\x\in\{-\frac{1}{\sqrt{n}},\frac{1}{\sqrt{n}}\}^n}(e^{c_3\h^T\x})\max_{\|\y\|_2=1}(e^{c_3\g^T\y}))\\
=E(\max_{\x\in\{-\frac{1}{\sqrt{n}},\frac{1}{\sqrt{n}}\}^n}(e^{c_3\h^T\x}))E(\max_{\|\y\|_2=1}(e^{c_3\g^T\y})),\label{eq:chpos2}
\end{multline}
where the last equality follows because of independence of $\g$ and $\h$. Connecting beginning and end of (\ref{eq:chpos2}) one has
\begin{equation}
E(e^{c_3g})E(e^{c_3(\max_{\x\in\{-\frac{1}{\sqrt{n}},\frac{1}{\sqrt{n}}\}^n}\|H\x\|_2)})\leq
E(\max_{\x\in\{-\frac{1}{\sqrt{n}},\frac{1}{\sqrt{n}}\}^n}(e^{c_3\h^T\x}))E(\max_{\|\y\|_2=1}(e^{c_3\g^T\y})).\label{eq:chpos3}
\end{equation}
Applying $\log$ on both sides of (\ref{eq:chpos3}) we further have
\begin{equation}
\log(E(e^{c_3g}))+\log(E(e^{c_3(\max_{\x\in\{-\frac{1}{\sqrt{n}},\frac{1}{\sqrt{n}}\}^n}\|H\x\|_2)}))\leq
\log(E(\max_{\x\in\{-\frac{1}{\sqrt{n}},\frac{1}{\sqrt{n}}\}^n}(e^{c_3\h^T\x})))+\log(E(\max_{\|\y\|_2=1}(e^{c_3\g^T\y}))),\label{eq:chpos4}
\end{equation}
or in a slightly more convenient form
\begin{equation}
\log(E(e^{c_3(\max_{\x\in\{-\frac{1}{\sqrt{n}},\frac{1}{\sqrt{n}}\}^n}\|H\x\|_2)}))\leq
-\log(E(e^{c_3g}))+\log(E(\max_{\x\in\{-\frac{1}{\sqrt{n}},\frac{1}{\sqrt{n}}\}^n}(e^{c_3\h^T\x})))+\log(E(\max_{\|\y\|_2=1}(e^{c_3\g^T\y}))).\label{eq:chpos5}
\end{equation}
It is also relatively easy to see that
\begin{equation}
\log(E(e^{c_3(\max_{\x\in\{-\frac{1}{\sqrt{n}},\frac{1}{\sqrt{n}}\}^n}\|H\x\|_2)}))\geq
E\log(e^{c_3(\max_{\x\in\{-\frac{1}{\sqrt{n}},\frac{1}{\sqrt{n}}\}^n}\|H\x\|_2)})=Ec_3(\max_{\x\in\{-\frac{1}{\sqrt{n}},\frac{1}{\sqrt{n}}\}^n}\|H\x\|_2),
\label{eq:chpos6}
\end{equation}
and
\begin{equation}
-\log(E(e^{c_3g}))=-\log(e^{\frac{c_3^2}{2}})=-\frac{c_3^2}{2}.\label{eq:chpos7}
\end{equation}
Connecting (\ref{eq:chpos5}), (\ref{eq:chpos6}), and (\ref{eq:chpos7}) one finally can establish an upper bound on the expected value of the ground state energy of the positive Hopfiled model
\begin{equation}
E(\max_{\x\in\{-\frac{1}{\sqrt{n}},\frac{1}{\sqrt{n}}\}^n}\|H\x\|_2)\leq
-\frac{c_3}{2}+\frac{1}{c_3}\log(E(\max_{\x\in\{-\frac{1}{\sqrt{n}},\frac{1}{\sqrt{n}}\}^n}(e^{c_3\h^T\x})))
+\frac{1}{c_3}\log(E(\max_{\|\y\|_2=1}(e^{c_3\g^T\y}))).\label{eq:chpos8}
\end{equation}
Let $c_3=c_3^{(s)}\sqrt{n}$ where $c_3^{(s)}$ is a constant independent of $n$. Then (\ref{eq:chpos8}) becomes
\begin{eqnarray}
\hspace{-.5in}\frac{E(\max_{\x\in\{-\frac{1}{\sqrt{n}},\frac{1}{\sqrt{n}}\}^n}\|H\x\|_2)}{\sqrt{n}}
& \leq &
-\frac{c_3^{(s)}}{2}+\frac{1}{nc_3^{(s)}}\log(E(\max_{\x\in\{-1,1\}^n}(e^{c_3^{(s)}\h^T\x})))
+\frac{1}{nc_3^{(s)}}\log(E(\max_{\|\y\|_2=1}(e^{c_3^{(s)}\sqrt{n}\g^T\y})))\nonumber \\
& = &
-\frac{c_3^{(s)}}{2}+\frac{1}{c_3^{(s)}}\log(E(e^{c_3^{(s)}|\h_1|}))
+\frac{1}{nc_3^{(s)}}\log(E(\max_{\|\y\|_2=1}(e^{c_3^{(s)}\sqrt{n}\g^T\y})))\nonumber \\
& = &
-\frac{c_3^{(s)}}{2}+\frac{c_3^{(s)}}{2}+\frac{1}{c_3^{(s)}}\log(\mbox{erfc}(-\frac{c_3^{(s)}}{\sqrt{2}}))
+\frac{1}{nc_3^{(s)}}\log(E(\max_{\|\y\|_2=1}(e^{c_3^{(s)}\sqrt{n}\g^T\y})))\nonumber \\
& = & \frac{1}{c_3^{(s)}}\log(\mbox{erfc}(-\frac{c_3^{(s)}}{\sqrt{2}}))
+\frac{1}{nc_3^{(s)}}\log(E(\max_{\|\y\|_2=1}(e^{c_3^{(s)}\sqrt{n}\g^T\y}))).\label{eq:chpos9}
\end{eqnarray}
One should now note that the above bound is effectively correct for any positive constant $c_3^{(s)}$. The only thing that is then left to be done so that the above bound becomes operational is to estimate $E(\max_{\|\y\|_2=1}(e^{c_3^{(s)}\sqrt{n}\g^T\y}))=Ee^{c_3^{(s)}\sqrt{n}\|\g\|_2}$. Pretty good estimates for this quantity can be obtained for any $n$. However, to facilitate the exposition we will focus only on the large $n$ scenario. In that case one can use the saddle point concept applied in \cite{SPH}. However, here we will try to avoid the entire presentation from there and instead present the core neat idea that has much wider applications. Namely, we start with the following identity
\begin{equation}
\|\g\|_2=\min_{\gamma\geq 0}(\frac{\|\g\|_2^2}{4\gamma}+\gamma).\label{eq:gamaiden}
\end{equation}
Then
\begin{multline}
\frac{1}{nc_3^{(s)}}\log(Ee^{c_3^{(s)}\sqrt{n}\|\g\|_2})=\frac{1}{nc_3^{(s)}}\log(Ee^{c_3^{(s)}\sqrt{n}\min_{\gamma\geq 0}(\frac{\|\g\|_2^2}{4\gamma}+\gamma)})
\doteq \frac{1}{nc_3^{(s)}}\min_{\gamma\geq 0}\log(Ee^{c_3^{(s)}\sqrt{n}(\frac{\|\g\|_2^2}{4\gamma}+\gamma)})\\
=\min_{\gamma\geq 0}(\frac{\gamma}{\sqrt{n}}+\frac{1}{c_3^{(s)}}\log(Ee^{c_3^{(s)}\sqrt{n}(\frac{\g_i^2}{4\gamma})})),\label{eq:gamaiden1}
\end{multline}
where $\doteq$ stands for equality when $n\rightarrow \infty$. $\doteq$ is exactly what was shown in \cite{SPH}. In fact, a bit more is shown in \cite{SPH} and a few corrective terms were estimated for finite $n$ (for our needs here though, even just replacing $\doteq$ with $\leq$ inequality suffices). Now if one sets $\gamma=\gamma^{(s)}\sqrt{n}$ then (\ref{eq:gamaiden1}) gives
\begin{equation}
\frac{1}{nc_3^{(s)}}\log(Ee^{c_3^{(s)}\sqrt{n}\|\g\|_2})
=\min_{\gamma^{(s)}\geq 0}(\gamma^{(s)}+\frac{1}{c_3^{(s)}}\log(Ee^{c_3^{(s)}(\frac{\g_i^2}{4\gamma^{(s)}})}))
=\min_{\gamma^{(s)}\geq 0}(\gamma^{(s)}-\frac{\alpha}{2c_3^{(s)}}\log(1-\frac{c_3^{(s)}}{2\gamma^{(s)}})).\label{eq:gamaiden2}
\end{equation}
After solving the last minimization one obtains
\begin{equation}
\widehat{\gamma^{(s)}}=\frac{2c_3^{(s)}+\sqrt{4(c_3^{(s)})^2+16\alpha}}{8}.\label{eq:gamaiden3}
\end{equation}
Connecting (\ref{eq:chpos9}), (\ref{eq:gamaiden1}), (\ref{eq:gamaiden2}), and (\ref{eq:gamaiden3}) one finally has
\begin{equation}
\hspace{-.5in}\frac{E(\max_{\x\in\{-\frac{1}{\sqrt{n}},\frac{1}{\sqrt{n}}\}^n}\|H\x\|_2)}{\sqrt{n}}
 \leq  \frac{1}{c_3^{(s)}}\log(\mbox{erfc}(-\frac{c_3^{(s)}}{\sqrt{2}}))
+\widehat{\gamma^{(s)}}-\frac{\alpha}{2c_3^{(s)}}\log(1-\frac{c_3^{(s)}}{2\widehat{\gamma^{(s)}}}),\label{eq:ubmorsoph}
\end{equation}
where clearly $\widehat{\gamma^{(s)}}$ is as in (\ref{eq:gamaiden3}). As mentioned earlier the above inequality holds for any $c_3^{(s)}$. Of course to make it as tight as possible one then has
\begin{equation}
\hspace{-.5in}\frac{E(\max_{\x\in\{-\frac{1}{\sqrt{n}},\frac{1}{\sqrt{n}}\}^n}\|H\x\|_2)}{\sqrt{n}}
 \leq  \min_{c_3^{(s)}\geq 0} \left (\frac{1}{c_3^{(s)}}\log(\mbox{erfc}(-\frac{c_3^{(s)}}{\sqrt{2}}))
+\widehat{\gamma^{(s)}}-\frac{\alpha}{2c_3^{(s)}}\log(1-\frac{c_3^{(s)}}{2\widehat{\gamma^{(s)}}}\right ).\label{eq:ubmorsoph1}
\end{equation}

We summarize our results from this subsection in the following lemma.

\begin{lemma}
Let $H$ be an $m\times n$ matrix with i.i.d. standard normal components. Let $n$ be large and let $m=\alpha n$, where $\alpha>0$ is a constant independent of $n$. Let $\xi_p$ be as in (\ref{eq:sqrtposham1}). Let $\widehat{\gamma^{(s)}}$ be such that
\begin{equation}
\widehat{\gamma^{(s)}}=\frac{2c_3^{(s)}+\sqrt{4(c_3^{(s)})^2+16\alpha}}{8}.\label{eq:gamaiden3thm}
\end{equation}
and $\xi_p^{(u)}$ be a scalar such that
\begin{equation}
\xi_p^{(u)}=\min_{c_3^{(s)}\geq 0} \left (\frac{1}{c_3^{(s)}}\log(\mbox{erfc}(-\frac{c_3^{(s)}}{\sqrt{2}}))
+\widehat{\gamma^{(s)}}-\frac{\alpha}{2c_3^{(s)}}\log(1-\frac{c_3^{(s)}}{2\widehat{\gamma^{(s)}}}\right ).\label{eq:condxipuposgenlemma}
\end{equation}
Then
\begin{equation}
\frac{E\xi_p}{\sqrt{n}}\leq\xi_p^{(u)}.\label{eq:posgenexplemma}
\end{equation}
Moreover,
\begin{eqnarray}
& & \lim_{n\rightarrow\infty}P(\max_{\x\in\{-\frac{1}{\sqrt{n}},\frac{1}{\sqrt{n}}\}^n}(\|H\x\|_2)\leq \xi_p^{(u)})\geq 1\nonumber \\
& \Leftrightarrow & \lim_{n\rightarrow\infty}P(\xi_p\leq \xi_p^{(u)})\geq 1 \nonumber \\
& \Leftrightarrow & \lim_{n\rightarrow\infty}P(\xi_p^2\leq (\xi_p^{(u)})^2)\geq 1. \label{eq:posgenproblemma}
\end{eqnarray}
In particular, when $\alpha=1$
\begin{equation}
\frac{E\xi_p}{\sqrt{n}}\leq\xi_p^{(u)}=1.7832.\label{eq:posgenexplemma1}
\end{equation}
\label{lemma:posgenlemma}
\end{lemma}
\begin{proof}
The first part of the proof related to the expected values follows from the above discussion. The probability part follows by the concentration arguments that are easy to establish (see, e.g discussion in \cite{StojnicHopBnds10}).
\end{proof}

One way to see how the above lemma works in practice is (as specified in the lemma) to choose $\alpha=1$ to obtain $\xi_p^{(u)}=1.7832$. This value is substantially better than $1.7978$ offered in \cite{StojnicHopBnds10}.

\section{Negative Hopfield form}
\label{sec:neghop}

In this section we will look at the following optimization problem (which clearly is the key component in estimating the ground state energy in the thermodynamic limit)
\begin{equation}
\min_{\x\in\{-\frac{1}{\sqrt{n}},\frac{1}{\sqrt{n}}\}^n}\|H\x\|_2^2.\label{eq:negham1}
\end{equation}
For a deterministic (given fixed) $H$ this problem is of course known to be NP-hard (as (\ref{eq:posham1}), it essentially falls under the class of binary quadratic optimization problems). Instead of looking at the problem in (\ref{eq:negham1}) in a deterministic way i.e. in a way that assumes that matrix $H$ is deterministic, we will adopt the strategy of the previous section and look at it in a statistical scenario. Also as in previous section, we will assume that the elements of matrix $H$ are i.i.d. standard normals. In the remainder of this section we will look at possible ways to estimate the optimal value of the optimization problem in (\ref{eq:negham1}). In fact we will introduce a strategy similar the one presented in the previous section to create a lower-bound on the optimal value of (\ref{eq:negham1}).

\subsection{Lower-bounding ground state energy of the negative Hopfield form}
\label{sec:neghoplb}

In this section we will look at problem from (\ref{eq:negham1}).In fact, to be a bit more precise, as in the previous section, in order to make the exposition as simple as possible, we will look at its a slight variant given below
\begin{equation}
\xi_n=\min_{\x\in\{-\frac{1}{\sqrt{n}},\frac{1}{\sqrt{n}}\}^n}\|H\x\|_2.\label{eq:sqrtnegham1}
\end{equation}
As mentioned above, we will assume that the elements of $H$ are i.i.d. standard normal random variables.

First we recall (and slightly extend) the following result from \cite{Gordon85} that relates to statistical properties of certain Gaussian processes. This result is essentially a negative counterpart to the one given in Theorem \ref{thm:Gordonpos1}
\begin{theorem}(\cite{Gordon85})
\label{thm:Gordonneg1} Let $X_{ij}$ and $Y_{ij}$, $1\leq i\leq n,1\leq j\leq m$, be two centered Gaussian processes which satisfy the following inequalities for all choices of indices
\begin{enumerate}
\item $E(X_{ij}^2)=E(Y_{ij}^2)$
\item $E(X_{ij}X_{ik})\geq E(Y_{ij}Y_{ik})$
\item $E(X_{ij}X_{lk})\leq E(Y_{ij}Y_{lk}), i\neq l$.
\end{enumerate}
Let $\psi()$ be an increasing function on the real axis. Then
\begin{equation*}
E(\min_{i}\max_{j}\psi(X_{ij}))\leq E(\min_{i}\max_{j}\psi(Y_{ij})).
\end{equation*}
Moreover, let $\psi()$ be a decreasing function on the real axis. Then
\begin{equation*}
E(\max_{i}\min_{j}\psi(X_{ij}))\geq E(\max_{i}\min_{j}\psi(Y_{ij})).
\end{equation*}
\begin{proof}
The proof of all statements but the last one is of course given in \cite{Gordon85}. Here we just briefly sketch how to get the last statement as well. So, let $\psi()$ be a decreasing function on the real axis. Then $-\psi()$ is an increasing function on the real axis and by first part of the theorem we have
\begin{equation*}
E(\min_{i}\max_{j}-\psi(X_{ij}))\leq E(\min_{i}\max_{j}-\psi(Y_{ij})).
\end{equation*}
Changing the inequality sign we also have
\begin{equation*}
-E(\min_{i}\max_{j}-\psi(X_{ij}))\geq -E(\min_{i}\max_{j}-\psi(Y_{ij})),
\end{equation*}
and finally
\begin{equation*}
E(\max_{i}\min_{j}\psi(X_{ij}))\geq E(\max_{i}\min_{j}\psi(Y_{ij})).
\end{equation*}
\end{proof}
\end{theorem}

In our recent work \cite{StojnicHopBnds10} we rely on the above theorem to create a lower-bound on the ground state energy of the negative Hopfield model. However, as was the case with the positive form, the strategy employed in \cite{StojnicHopBnds10} relied only on a basic version of the above theorem where $\psi(x)=x$. Similarly to what was done in the previous subsection, we will here substantially upgrade the strategy from \cite{StojnicHopBnds10} by looking at a very simple (but way better) different version of $\psi()$.

We start by reformulating the problem in (\ref{eq:sqrtposham1}) in the following way
\begin{equation}
\xi_n=\min_{\x\in\{-\frac{1}{\sqrt{n}},\frac{1}{\sqrt{n}}\}^n}\max_{\|\y\|_2=1}\y^TH\x.\label{eq:sqrtnegham2}
\end{equation}
As was the case with the positive form, we do mention without going into details that the ground state energies will again concentrate in the thermodynamic limit and hence we will mostly focus on the expected value of $\xi_n$ (one can then easily adapt our results to describe more general probabilistic concentrating properties of ground state energies). The following is then a direct application of Theorem \ref{thm:Gordonneg1}.
\begin{lemma}
Let $H$ be an $m\times n$ matrix with i.i.d. standard normal components. Let $\g$ and $\h$ be $m\times 1$ and $n\times 1$ vectors, respectively, with i.i.d. standard normal components. Also, let $g$ be a standard normal random variable and let $c_3$ be a positive constant. Then
\begin{equation}
E(\max_{\x\in\{-\frac{1}{\sqrt{n}},\frac{1}{\sqrt{n}}\}^n}\min_{\|\y\|_2=1}e^{-c_3(\y^T H\x + g)})\leq E(\max_{\x\in\{-\frac{1}{\sqrt{n}},\frac{1}{\sqrt{n}}\}^n}\min_{\|\y\|_2=1}e^{-c_3(\g^T\y+\h^T\x)}).\label{eq:negexplemma}
\end{equation}\label{lemma:negexplemma}
\end{lemma}
\begin{proof}
As mentioned above, the proof is a standard/direct application of Theorem \ref{thm:Gordonneg1}. We will sketch it for completeness. Namely, one starts by defining processes $X_i$ and $Y_i$ in the following way
\begin{equation}
Y_{ij}=(\y^{(j)})^T H\x^{(i)} + g\quad X_{ij}=\g^T\y^{(j)}+\h^T\x^{(i)}.\label{eq:negexplemmaproof1}
\end{equation}
Then clearly
\begin{equation}
EY_{ij}^2=EX_{ij}^2=2.\label{eq:negexplemmaproof2}
\end{equation}
One then further has
\begin{eqnarray}
EY_{ij}Y_{ik} & = & (\y^{(k)})^T\y^{(j)}+1 \nonumber \\
EX_{ij}X_{ik} & = & (\y^{(k)})^T\y^{(j)}+1,\label{eq:negexplemmaproof3}
\end{eqnarray}
and clearly
\begin{equation}
EX_{ij}X_{ik}=EY_{ij}Y_{ik}.\label{eq:negexplemmaproof31}
\end{equation}
Moreover,
\begin{eqnarray}
EY_{ij}Y_{lk} & = & (\y^{(j)})^T\y^{(k)}(\x^{(i)})^T\x^{(l)}+1 \nonumber \\
EX_{ij}X_{lk} & = & (\y^{(j)})^T\y^{(k)}+(\x^{(i)})^T\x^{(l)}.\label{eq:negexplemmaproof32}
\end{eqnarray}
And after a small algebraic transformation
\begin{eqnarray}
EY_{ij}Y_{lk}-EX_{ij}X_{lk} & = & (1-(\y^{(j)})^T\y^{(k)})-(\x^{(i)})^T\x^{(l)}(1-(\y^{(j)})^T\y^{(k)}) \nonumber \\
& = & (1-(\x^{(i)})^T\x^{(l)})(1-(\y^{(j)})^T\y^{(k)})\nonumber \\
& \geq & 0.\label{eq:negexplemmaproof4}
\end{eqnarray}
Combining (\ref{eq:negexplemmaproof2}), (\ref{eq:negexplemmaproof31}), and (\ref{eq:negexplemmaproof4}) and using results of Theorem \ref{thm:Gordonneg1} one then easily obtains (\ref{eq:negexplemma}).
\end{proof}

Following what was done in Subsection \ref{sec:poshopub} one then easily has
\begin{multline}
E(e^{-c_3(\min_{\x\in\{-\frac{1}{\sqrt{n}},\frac{1}{\sqrt{n}}\}^n}\|H\x\|_2+g)})=
E(e^{-c_3(\min_{\x\in\{-\frac{1}{\sqrt{n}},\frac{1}{\sqrt{n}}\}^n}\max_{\|\y\|_2=1}\y^TH\x+g)})\\
=E(\max_{\x\in\{-\frac{1}{\sqrt{n}},\frac{1}{\sqrt{n}}\}^n}\min_{\|\y\|_2=1}(e^{-c_3(\y^TH\x+g)})).\label{eq:chneg1}
\end{multline}
Connecting (\ref{eq:chneg1}) and results of Lemma \ref{lemma:negexplemma} we have
\begin{multline}
E(e^{-c_3(\min_{\x\in\{-\frac{1}{\sqrt{n}},\frac{1}{\sqrt{n}}\}^n}\|H\x\|_2+g)})=
E(\max_{\x\in\{-\frac{1}{\sqrt{n}},\frac{1}{\sqrt{n}}\}^n}\min_{\|\y\|_2=1}(e^{-c_3(\y^TH\x+g)}))\\
\leq E(\max_{\x\in\{-\frac{1}{\sqrt{n}},\frac{1}{\sqrt{n}}\}^n}\min_{\|\y\|_2=1}(e^{-c_3(\g^T\y+h^T\x)}))
=E(\max_{\x\in\{-\frac{1}{\sqrt{n}},\frac{1}{\sqrt{n}}\}^n}(e^{-c_3\h^T\x})\min_{\|\y\|_2=1}(e^{-c_3\g^T\y}))\\
=E(\max_{\x\in\{-\frac{1}{\sqrt{n}},\frac{1}{\sqrt{n}}\}^n}(e^{-c_3\h^T\x}))E(\min_{\|\y\|_2=1}(e^{-c_3\g^T\y})),\label{eq:chneg2}
\end{multline}
where the last equality follows because of independence of $\g$ and $\h$. Connecting beginning and end of (\ref{eq:chneg2}) one has
\begin{equation}
E(e^{-c_3g})E(e^{-c_3(\min_{\x\in\{-\frac{1}{\sqrt{n}},\frac{1}{\sqrt{n}}\}^n}\|H\x\|_2)})\leq
E(\max_{\x\in\{-\frac{1}{\sqrt{n}},\frac{1}{\sqrt{n}}\}^n}(e^{-c_3\h^T\x}))E(\min_{\|\y\|_2=1}(e^{-c_3\g^T\y})).\label{eq:chneg3}
\end{equation}
Applying $\log$ on both sides of (\ref{eq:chneg3}) we further have
\begin{equation}
\hspace{-.5in}\log(E(e^{-c_3g}))+\log(E(e^{-c_3(\min_{\x\in\{-\frac{1}{\sqrt{n}},\frac{1}{\sqrt{n}}\}^n}\|H\x\|_2)}))\leq
\log(E(\max_{\x\in\{-\frac{1}{\sqrt{n}},\frac{1}{\sqrt{n}}\}^n}(e^{-c_3\h^T\x})))+\log(E(\min_{\|\y\|_2=1}(e^{-c_3\g^T\y}))),\label{eq:chneg4}
\end{equation}
or in a slightly more convenient form
\begin{equation}
\hspace{-.5in}\log(E(e^{-c_3(\min_{\x\in\{-\frac{1}{\sqrt{n}},\frac{1}{\sqrt{n}}\}^n}\|H\x\|_2)}))\leq
-\log(E(e^{-c_3g}))+\log(E(\max_{\x\in\{-\frac{1}{\sqrt{n}},\frac{1}{\sqrt{n}}\}^n}(e^{-c_3\h^T\x})))+\log(E(\min_{\|\y\|_2=1}(e^{-c_3\g^T\y}))).\label{eq:chneg5}
\end{equation}
It is also relatively easy to see that
\begin{equation}
\log(E(e^{-c_3(\min_{\x\in\{-\frac{1}{\sqrt{n}},\frac{1}{\sqrt{n}}\}^n}\|H\x\|_2)}))\geq
E\log(e^{-c_3(\min_{\x\in\{-\frac{1}{\sqrt{n}},\frac{1}{\sqrt{n}}\}^n}\|H\x\|_2)})=-Ec_3(\min_{\x\in\{-\frac{1}{\sqrt{n}},\frac{1}{\sqrt{n}}\}^n}\|H\x\|_2),
\label{eq:chneg6}
\end{equation}
and as earlier
\begin{equation}
-\log(E(e^{-c_3g}))=-\log(e^{\frac{c_3^2}{2}})=-\frac{c_3^2}{2}.\label{eq:chneg7}
\end{equation}
Connecting (\ref{eq:chneg5}), (\ref{eq:chneg6}), and (\ref{eq:chneg7}) one finally can establish a lower bound on the expected value of the ground state energy of the negative Hopfiled model
\begin{equation}
E(\min_{\x\in\{-\frac{1}{\sqrt{n}},\frac{1}{\sqrt{n}}\}^n}\|H\x\|_2)\geq
\frac{c_3}{2}-\frac{1}{c_3}\log(E(\max_{\x\in\{-\frac{1}{\sqrt{n}},\frac{1}{\sqrt{n}}\}^n}(e^{-c_3\h^T\x})))
-\frac{1}{c_3}\log(E(\min_{\|\y\|_2=1}(e^{-c_3\g^T\y}))).\label{eq:chneg8}
\end{equation}
Let $c_3=c_3^{(s)}\sqrt{n}$ where $c_3^{(s)}$ is a constant independent of $n$. Then (\ref{eq:chneg8}) becomes
\begin{eqnarray}
\hspace{-.5in}\frac{E(\min_{\x\in\{-\frac{1}{\sqrt{n}},\frac{1}{\sqrt{n}}\}^n}\|H\x\|_2)}{\sqrt{n}}
& \geq &
\frac{c_3^{(s)}}{2}-\frac{1}{nc_3^{(s)}}\log(E(\max_{\x\in\{-1,1\}^n}(e^{-c_3^{(s)}\h^T\x})))
-\frac{1}{nc_3^{(s)}}\log(E(\min_{\|\y\|_2=1}(e^{-c_3^{(s)}\sqrt{n}\g^T\y})))\nonumber \\
& = &
\frac{c_3^{(s)}}{2}-\frac{1}{c_3^{(s)}}\log(E(e^{c_3^{(s)}|\h_1|}))
-\frac{1}{nc_3^{(s)}}\log(E(\min_{\|\y\|_2=1}(e^{-c_3^{(s)}\sqrt{n}\g^T\y})))\nonumber \\
& = &
\frac{c_3^{(s)}}{2}-\frac{c_3^{(s)}}{2}-\frac{1}{c_3^{(s)}}\log(\mbox{erfc}(-\frac{c_3^{(s)}}{\sqrt{2}}))
-\frac{1}{nc_3^{(s)}}\log(E(\min_{\|\y\|_2=1}(e^{-c_3^{(s)}\sqrt{n}\g^T\y})))\nonumber \\
& = & -\frac{1}{c_3^{(s)}}\log(\mbox{erfc}(-\frac{c_3^{(s)}}{\sqrt{2}}))
-\frac{1}{nc_3^{(s)}}\log(E(\min_{\|\y\|_2=1}(e^{-c_3^{(s)}\sqrt{n}\g^T\y}))).\label{eq:chneg9lift}
\end{eqnarray}
One should now note that the above bound is effectively correct for any positive constant $c_3^{(s)}$. The only thing that is then left to be done so that the above bound becomes operational is to estimate $E(\min_{\|\y\|_2=1}(e^{-c_3^{(s)}\sqrt{n}\g^T\y}))=Ee^{-c_3^{(s)}\sqrt{n}\|\g\|_2}$. Pretty good estimates for this quantity can be obtained for any $n$. However, to facilitate the exposition we will focus only on the large $n$ scenario. Again, in that case one can use the saddle point concept applied in \cite{SPH}. However, as earlier, here we will try to avoid the entire presentation from there and instead present the core neat idea that has much wider applications. Namely, we start with the following identity
\begin{equation}
-\|\g\|_2=\max_{\gamma\geq 0}(-\frac{\|\g\|_2^2}{4\gamma}-\gamma).\label{eq:gamaiden}
\end{equation}
Then
\begin{multline}
\frac{1}{nc_3^{(s)}}\log(Ee^{-c_3^{(s)}\sqrt{n}\|\g\|_2})=\frac{1}{nc_3^{(s)}}\log(Ee^{c_3^{(s)}\sqrt{n}\max_{\gamma\geq 0}(-\frac{\|\g\|_2^2}{4\gamma}-\gamma)})
\doteq \frac{1}{nc_3^{(s)}}\max_{\gamma\geq 0}\log(Ee^{-c_3^{(s)}\sqrt{n}(\frac{\|\g\|_2^2}{4\gamma}+\gamma)})\\
=\max_{\gamma\geq 0}(-\frac{\gamma}{\sqrt{n}}+\frac{1}{c_3^{(s)}}\log(Ee^{-c_3^{(s)}\sqrt{n}(\frac{\g_i^2}{4\gamma})})),\label{eq:gamaiden1lift}
\end{multline}
where as earlier $\doteq$ stands for equality when $n\rightarrow \infty$. Also, as mentioned earlier, $\doteq$ is exactly what was shown in \cite{SPH}. Now if one sets $\gamma=\gamma^{(s)}\sqrt{n}$ then (\ref{eq:gamaiden1lift}) gives
\begin{multline}
\frac{1}{nc_3^{(s)}}\log(Ee^{-c_3^{(s)}\sqrt{n}\|\g\|_2})
=\max_{\gamma^{(s)}\geq 0}(-\gamma^{(s)}+\frac{1}{c_3^{(s)}}\log(Ee^{-c_3^{(s)}(\frac{\g_i^2}{4\gamma^{(s)}})}))
=\max_{\gamma^{(s)}\geq 0}(-\gamma^{(s)}-\frac{\alpha}{2c_3^{(s)}}\log(1+\frac{c_3^{(s)}}{2\gamma^{(s)}}))\\
=\max_{\gamma^{(s)}\leq 0}(\gamma^{(s)}-\frac{\alpha}{2c_3^{(s)}}\log(1-\frac{c_3^{(s)}}{2\gamma^{(s)}})).\label{eq:gamaiden2lift}
\end{multline}
After solving the last maximization one obtains
\begin{equation}
\widehat{\gamma_n^{(s)}}=\frac{2c_3^{(s)}-\sqrt{4(c_3^{(s)})^2+16\alpha}}{8}.\label{eq:gamaiden3lift}
\end{equation}
Connecting (\ref{eq:chneg9lift}), (\ref{eq:gamaiden1lift}), (\ref{eq:gamaiden2lift}), and (\ref{eq:gamaiden3lift}) one finally has
\begin{equation}
\hspace{-.5in}\frac{E(\min_{\x\in\{-\frac{1}{\sqrt{n}},\frac{1}{\sqrt{n}}\}^n}\|H\x\|_2)}{\sqrt{n}}
 \geq  -\left (\frac{1}{c_3^{(s)}}\log(\mbox{erfc}(-\frac{c_3^{(s)}}{\sqrt{2}}))
+\widehat{\gamma_n^{(s)}}-\frac{\alpha}{2c_3^{(s)}}\log(1-\frac{c_3^{(s)}}{2\widehat{\gamma_n^{(s)}}})\right ),\label{eq:ubmorsoph}
\end{equation}
where clearly $\widehat{\gamma^{(s)}}$ is as in (\ref{eq:gamaiden3}). As mentioned earlier the above inequality holds for any $c_3^{(s)}$. Of course to make it as tight as possible one then has
\begin{equation}
\hspace{-.5in}\frac{E(\min_{\x\in\{-\frac{1}{\sqrt{n}},\frac{1}{\sqrt{n}}\}^n}\|H\x\|_2)}{\sqrt{n}}
 \geq - \min_{c_3^{(s)}\geq 0} \left (\frac{1}{c_3^{(s)}}\log(\mbox{erfc}(-\frac{c_3^{(s)}}{\sqrt{2}}))
+\widehat{\gamma_n^{(s)}}-\frac{\alpha}{2c_3^{(s)}}\log(1-\frac{c_3^{(s)}}{2\widehat{\gamma_n^{(s)}}}\right ).\label{eq:ubmorsoph1}
\end{equation}

We summarize our results from this subsection in the following lemma.

\begin{lemma}
Let $H$ be an $m\times n$ matrix with i.i.d. standard normal components. Let $n$ be large and let $m=\alpha n$, where $\alpha>0$ is a constant independent of $n$. Let $\xi_n$ be as in (\ref{eq:sqrtnegham1}). Let $\widehat{\gamma_n^{(s)}}$ be such that
\begin{equation}
\widehat{\gamma_n^{(s)}}=\frac{2c_3^{(s)}-\sqrt{4(c_3^{(s)})^2+16\alpha}}{8}.\label{eq:gamaiden3thm}
\end{equation}
and $\xi_n^{(l)}$ be a scalar such that
\begin{equation}
\xi_n^{(l)}=-\min_{c_3^{(s)}\geq 0} \left (\frac{1}{c_3^{(s)}}\log(\mbox{erfc}(-\frac{c_3^{(s)}}{\sqrt{2}}))
+\widehat{\gamma^{(s)}}-\frac{\alpha}{2c_3^{(s)}}\log(1-\frac{c_3^{(s)}}{2\widehat{\gamma^{(s)}}}\right ).\label{eq:condxipuneggenlemma}
\end{equation}
Then
\begin{equation}
\frac{E\xi_n}{\sqrt{n}}\geq\xi_n^{(l)}.\label{eq:neggenexplemma}
\end{equation}
Moreover,
\begin{eqnarray}
& & \lim_{n\rightarrow\infty}P(\min_{\x\in\{-\frac{1}{\sqrt{n}},\frac{1}{\sqrt{n}}\}^n}(\|H\x\|_2)\geq \xi_n^{(l)})\geq 1\nonumber \\
& \Leftrightarrow & \lim_{n\rightarrow\infty}P(\xi_n\geq \xi_n^{(l)})\geq 1 \nonumber \\
& \Leftrightarrow & \lim_{n\rightarrow\infty}P(\xi_n^2\geq (\xi_n^{(l)})^2)\geq 1. \label{eq:neggenproblemma}
\end{eqnarray}
In particular, when $\alpha=1$
\begin{equation}
\frac{E\xi_n}{\sqrt{n}}\geq\xi_n^{(l)}=0.32016.\label{eq:neggenexplemma1}
\end{equation}
\label{lemma:neggenlemma}
\end{lemma}
\begin{proof}
The first part of the proof related to the expected values follows from the above discussion. The probability part follows by the concentration arguments that are easy to establish (see, e.g discussion in \cite{StojnicHopBnds10}).
\end{proof}

One way to see how the above lemma works in practice is (as specified in the lemma) to choose $\alpha=1$ to obtain $\xi_n^{(u)}=0.32016$. This value is substantially better than $0.2021$ offered in \cite{StojnicHopBnds10}.

\section{Practical algorithmic observations of Hopfield forms}
\label{sec:alghop}

We just briefly comment on the quality of the results obtained above when compared to their optimal counterparts. Of course, we do not know what the optimal values for ground state energies are. However, we conducted a solid set of numerical experiments using various implementations of bit flipping algorithms (we of course restricted our attention to $\alpha=1$ case). Our feeling is that both bounds provided in this paper are very close to the exact values. We believe that the exact value for $\lim_{n\rightarrow\infty}\frac{E\xi_p}{\sqrt{n}}$ is somewhere around $1.78$. On the other hand, we believe that the exact value for $\lim_{n\rightarrow\infty}\frac{E\xi_n}{\sqrt{n}}$ is somewhere around $0.328$.

Another observation is actually probably more important. In terms of the size of the problems, the limiting value seems to be approachable substantially faster for the negative form. Even for a fairly small size $n=50$ the optimal values are already approaching $0.34$ barrier on average. However, for the positive form even dimensions of several hundreds are not even remotely enough to come close to the optimal value (of course for larger dimensions we solved the problems only approximately, but the solutions were sufficiently far away from the bound that it was hard for us to believe that even the exact solution in those scenarios is anywhere close to it). Of course, positive form is naturally an easier problem but there is a price to pay for being easier. One then may say that one way the paying price reveals itself is the slow $n$ convergence of the limiting optimal values.

\section{Conclusion}
\label{sec:conc}

In this paper we looked at classic positive  and negative Hopfield forms and their behavior in the zero-temperature limit which essentially amounts to the behavior of their ground state energies. We introduced fairly powerful mechanisms that can be used to provide bounds of the ground state energies of both models.

To be a bit more specific, we first provided purely theoretical upper bounds on the expected values of the ground state energy of the positive Hopfield model. These bounds present a substantial improvement over the classical ones we presented in \cite{StojnicHopBnds10}. Moreover, they in a way also present the first set of rigorous theoretical results that emphasizes the combinatorial structure of the problem. Also, we do believe that in the most widely known/studied square case (i.e. $\alpha=1$) the bounds are fairly close to the optimal values.

We then translated our results related to the positive Hopfield form to the case of the negative Hopfield form. We again targeted the ground state regime and provided a theoretical lower bound for the expected behavior of the ground state energy. The bounds we obtained for the negative form are an even more substantial improvement over the classical corresponding ones we presented in \cite{StojnicHopBnds10}. In fact, we believe that the bounds for the negative form are very close to the optimal value.

As was the case in \cite{StojnicHopBnds10}, the purely theoretical results we presented are for the so-called Gaussian Hopfield models, whereas in reality often a binary Hopfield model can be preferred. However, all results that we presented can easily be extended to the case of binary Hopfield models (and for that matter to an array of other statistical models as well). There are many ways how this can be done. Instead of recalling on them here we refer to a brief discussion about it that we presented in \cite{StojnicHopBnds10}.

We should add that ther results we presented in \cite{StojnicHopBnds10} are tightly connected with the ones that can be obtained through the very popular replica methods from statistical physics. In fact, what was shown in \cite{StojnicHopBnds10} essentially provided a rigorous proof that the replica symmetry type of results are actually rigorous upper/lower bounds on the ground state energies of the positive/negative Hopfield forms. In that sense what we presented here essentially confirms that a similar set of bounds that can be obtained assuming a variant of the first level of symmetry breaking are also rigorous upper/lower bounds. Showing this is relatively simple but does require a bit of a technical exposition and we find it more appropriate to present it in a separate paper.

We also recall (as in \cite{StojnicHopBnds10}) that in this paper we were mostly concerned with the behavior of the ground state energies. A vast majority of our results can be translated to characterize the behavior of the free energy when viewed at any temperature. While such a translation does not require any further insights it does require paying attention to a whole lot of little details and we will present it elsewhere.

\begin{singlespace}
\bibliographystyle{plain}
\bibliography{MoreSophHopBndsRefs}
\end{singlespace}

\end{document}